\documentclass[12pt]{amsart}
\usepackage{amssymb,amsmath,amsthm}
\usepackage{setspace}
\usepackage{graphicx}
\newcommand{\shrinkmargins}[1]{
  \addtolength{\textheight}{#1\topmargin}
  \addtolength{\textheight}{#1\topmargin}
  \addtolength{\textwidth}{#1\oddsidemargin}
  \addtolength{\textwidth}{#1\evensidemargin}
  \addtolength{\topmargin}{-#1\topmargin}
  \addtolength{\oddsidemargin}{-#1\oddsidemargin}
  \addtolength{\evensidemargin}{-#1\evensidemargin}
  }
\shrinkmargins{0.5}

\newtheorem{theorem}{Theorem}
\newtheorem{lemma}[theorem]{Lemma}
\newtheorem{corollary}[theorem]{Corollary}

\newtheorem{proposition}[theorem]{Proposition}
\newtheorem*{definition}{Definition}

\theoremstyle{remark}

\newtheorem*{remark}{Remark}

\numberwithin{theorem}{section} \numberwithin{equation}{section}

\newcommand{\Q}{\mathbb{Q}}

\newcommand{\N}{\mathbb{N}}

\begin{document}
\title[Records on the vanishing of Fourier coefficients of \dots ]%
{Records on the vanishing of Fourier coefficients of Powers Of the Dedekind Eta Function}
\author{Bernhard Heim }
\address{German University of Technology in Oman, Muscat, Sultanate of Oman}
\email{bernhard.heim@gutech.edu.om}
\author{Markus Neuhauser}
\address{German University of Technology in Oman, Muscat, Sultanate of Oman}
\email{markus.neuhauser@gutech.edu.om}
\author{Alexander Weisse}
\address{Max-Planck-Institute for Mathematics, Vivatsgasse 7, 53111 Bonn, Germany}
\email{weisse@mpim-bonn.mpg.de}
\subjclass[2010] {Primary 05A17, 11F20; Secondary 11F30, 11F37}
\keywords{Fourier Coefficients, Euler Products, Dedekind Eta Function, Lehmer Conjecture, Maeda conjecture}
\begin{abstract}
In this paper we significantly extend Serre's table on the vanishing
properties of Fourier coefficients of odd powers of the Dedekind eta function.
We address several conjectures of Cohen and Str{\"o}mberg and give a partial answer to a question of Ono.
In the even-power case, we extend Lehmer's conjecture on the coefficients of the discriminant function $\Delta$ to all non-CM-forms.
All our results are supported with numerical data. For example all Fourier coefficients $a_9(n)$ of the $9$-th power of the
Dedekind eta function are non-vanishing for $n \leq 10^{10}$. We also relate the non-vanishing of the Fourier coefficients of $\Delta^2$ to Maeda's conjecture.
\end{abstract}
\maketitle
\section{Introduction}
Already Euler and Jacobi had been interested in finding identities between infinite products and infinite sums,
to obtain information on the arithmetic properties of the coefficients.
\begin{eqnarray}
\sum_{n=0}^{\infty} a_1(n) \,\, X^n := \prod_{n=1}^{\infty} \left( 1 - X^n \right) & = &
\sum_{n= -\infty}^{\infty} (-1)^n \,\, X^{\frac{3n^2 +n}{2}},\label{Euler}\\
\sum_{n=0}^{\infty} a_3(n) \,\, X^n := \prod_{n=1}^{\infty} \left( 1 - X^n \right)^3 & = &
\sum_{n= 0}^{\infty} (-1)^n \,\, (2n+1) \,\, X^{\frac{n^2 +n}{2}}. \label{Jacobi}
\end{eqnarray}
These are the first examples of identities involving integral powers $\eta^r$ of the Dedekind eta function. Let
$\tau$ be in the upper half plane $\mathbb{H} := \{ z \in \mathbb{C} \, \vert \, \mathop{\rm Im}(z)>0 \}$ and $q:= e^{2 \pi \, i \tau}$.
In 1877 Dedekind introduced a modular form of weight $\frac{1}{2}$
\begin{equation}
\eta(\tau):= q^{\frac{1}{24}} \prod_{n=1}^{\infty} \left( 1 - q^n \right),
\end{equation}
a function on $\mathbb{H}$ satisfying
\begin{equation*}
\eta(\tau +1) =  e^{\pi \, i/12} \,\, \eta(\tau), \,\,\,
\eta  \left(-\frac{1}{\tau}\right) = \sqrt{-i \tau} \,\,  \eta(\tau).
\end{equation*}
The arithmetic properties, in particular the non-vanishing properties, of the Fourier coefficients $a_r(n)$, defined by
\begin{equation}
\eta(\tau)^r :=  q^{\frac{r}{24} }     \prod_{n=1}^{\infty}  \left( 1 - q^n \right)^r =
q^{\frac{r}{24} }     \sum_{n=0}^{\infty} a_r(n) \, q^n,
\end{equation}
are of importance.
In this paper we focus on $\eta^r$ for integers $r>0$, since for integers $r\leq 0$ the coefficients are either trivial (case $r=0$) or all strictly positive.

The general picture on the vanishing properties is as follows. It is known that $\eta^r$ is superlacunary \cite{OS95} if and only if
$ r \in S_{\text{odd}}:= \{ 1,3\}$ (the cases already considered by Euler and Jacobi). In the case that $r$ is even, Serre \cite{Se85} proved that $\eta(\tau)^r$ is lacunary, i.e.
\begin{eqnarray}
\lim_{N \to \infty} \frac{  \left|
\left\{ n \in \mathbb{N} \, \, \vert \,\, n \leq N, \, a(n) \neq 0 \right\}
\right| }{N} = 0,
\end{eqnarray}
if and only if $ r \in
S_{\text{even}}:=  \{ 2,4,6,8,10,14,26\}$. For $r=24$ the Lehmer conjecture predicts that the coefficients $\tau(n):=a_{24}(n-1) \neq 0$.
Ono \cite{On95} found that the case $r=12$ seems to have similar properties. For example the index $n$ of the first vanishing coefficient (if there is one) has to be a prime number.
In the odd case Serre published a table, based on partly unpublished results of Atkin, Cohen and Newman, of $n$ such that $a_r(n)=0$.
\\   \
\begin{center}
\fbox{%
\begin{tabular}{lll}
Atkin, Cohen & $r=5$ & $n=1560,1802,1838,2318,2690,\ldots $
\\
Atkin & $r=7$ & $n=28017$
\\
Newman \cite{Ne56} & $r=15$ & $n=53$
\end{tabular}
}
\end{center}
\ \\
This is an direct extract of
\cite{Se85}.
It is not mentioned how many pairs $(r,n)$ had been studied.
Additionally \cite{Co82} refers to results of Atkin and Newman, where $n$'s have been found for $r=9,11$ such that $a_r(n)=0$.
In this paper we could not confirm this result, since $a_r(n) \neq 0$ for these $r$'s and $n \leq 10^{10}$.

In \cite{HNR17}, we showed, for $r=9,11,13,17,19, 21,  23$, that $a_r(n)  \neq 0$ for all $n \leq 50 000$.
From a bird's eye view, having explicit formulas in mind, Cohen and Str{\"o}mberg (\cite{CS17} Remark 2.1.27) state that except for the cases
$r\in S_{\text{odd}}$ and $r\in S_{\text{even}}$, nothing is known,
but they give several conjectures (Exercise 2.6). They ask whether $\eta^5,\eta^{15}$ and $\eta^7$ have infinitely many vanishing coefficients,
and about their vanishing asymptotics, considering $r=5$ and $r=15$ to have similar properties.

We split our results into two parts.
Section 2 is devoted to the odd powers and section 3 to the even powers of the Dedekind eta function.
We give an interpretation of Serre's original table and extend it in the $n$ and $r$ aspect, studying the equation $a_r(n)=0$.
It turns out that the concept of sources, related to the Hecke theory of modular forms of half-integral weight is very useful.
We also give a partial answer to the conjectures of Cohen and Str{\"o}mberg. In section 3 we study the even case, which is related to
modular forms of integral weight. We indicate that Lehmer's conjecture has some chance of being true for all non-CM forms among the $\eta^r$.
Finally we show that the non-vanishing of the Fourier coefficients of the square of the discriminant function  $\Delta^2$ is directly related
to Maeda's conjecture.


\section{On the  odd powers of the Dedekind eta function}
In this section we extend Serre's table in the $r$ and $n$ aspect.
We refine a conjecture of Cohen and Str{\"o}mberg and give asymptotics for
$$\left|   \{ n \leq \, N \vert \,\,  a_r(n) \neq 0 \} \right| .
$$
Throughout this section, let $r$ be an odd positive integer.
We summarise our findings briefly as follows. The cases $r=1,3$ are completely understood by the formulas of Euler and Jacobi.
There are infinitely many
$n \in \mathbb{N}$ satisfying $a_r(n) =0$ for each $r=5,7$ and $15$.
For $n \leq 10^8$ there are no other $r<550$ such that $a_r(n) = 0$. The same is true for $n \leq 10^{10}$ in the cases
$r=9,11,13$. We introduce the concept of sources, based on Hecke theory for modular forms of half-integral weight. We show that for
$n\leq 10^{10}$, $r=5$ has $6352$ sources. 
For each $r=7$ and $r=15$ there is only one source for $n \leq 10^{10}$.

\subsection{Hecke theory of modular forms of half-integral weight}
Let us first recall some basic properties of the Hecke theory of modular forms of
half-integral weight, mainly due to Shimura \cite{Sh73}.
We follow the exposition given by Ono \cite{On03}.
Consider $\lambda \in \mathbb{N}_0$, $N$ a positive integer and $\chi$ a Dirichlet character modulo $4N$.
Let $\Gamma_0(m) := \{ \left(\begin{smallmatrix}    a & b \\ m c & d      \end{smallmatrix}\right) \in SL_2(\mathbb{Z}) \}$, $m \in \mathbb{N}$.
Then we denote by $S_{\lambda + \frac{1}{2}} \left( \Gamma_0(4N), \chi \right)$ the $\mathbb{C}$-vector space of weight
$\lambda + \frac{1}{2}$ modular cusp forms on $\Gamma_0(4N)$ with Nebentypus $\chi$.
It is well known that
$$\eta(24 \tau) \in S_{\frac{1}{2}} \left( \Gamma_0(4 \cdot\, 2^4 3^2), \chi_{12} \right).$$
Here $\chi_{12}(n)= 1$ if $n \equiv 1,11 \pmod{12}$, $-1$ if $n \equiv 5,7 \pmod{12}$, and $0$ otherwise.
Suppose that
$$ f(\tau) = \sum_{D=1}^{\infty} \, b(D) \, q^D  \in S_{\lambda + \frac{1}{2}} \left( \Gamma_0(4N), \chi \right).$$
Shimura introduced the Hecke operators $T(p^2,\lambda,\chi)$ for prime numbers $p$, given by
\begin{equation}
T(f(\tau)) :=  \sum_{D=1}^{\infty} \left(
b(Dp^2) + \chi^{*}(p)  \left(\frac{D}{p}\right) p^{\lambda -1} b(D) + \chi^{*}(p^2) p^{2 \lambda -1} b(D / p^2 ) \right) q^D.
\end{equation}
Here $\chi^{*}(D)  := \left(\frac{(-1)^{\lambda}}{D}\right)\chi(D)$, $b(l):= 0 $ if
$l \notin \mathbb{N}$, and $ \left(\frac{\phantom{a}}{\phantom{b}}\right) $
is the generalized Legendre symbol. We recall \cite{On03}, Proposition 3.46.
Let $f$ be as above. Suppose that $f$ is an eigenform of the Hecke operators $T(p^2,\lambda,\chi)$ for
primes $p \nmid 4N$. If $\lambda(p)$ are the eigenvalues, then
\begin{equation}\label{fundamental}
b(Dp^2)  =
\left( \lambda(p) -
\chi^{*}(p)  \left(\frac{D}{p}\right) p^{\lambda -1}\right) b(D) - \chi^{*}(p^2) p^{2 \lambda -1} b(D/p^2 )  .
\end{equation}
From (\ref{fundamental}) we conclude the following.
\begin{lemma}\label{eigenforms}
Let $p$ be an odd prime and let $(p,N)=1$. Let $D_0 \in \N$, where $p^2 \nmid D_0$. Then
\begin{equation}
b(D_0) = 0 \Longrightarrow b(D_0 n^2) = 0
\end{equation}
for all $n \in \mathbb{N}$ with $(n, 4N)=1$.
\end{lemma}

Let $f_r(\tau):= \eta(24 \tau)^r$, with Fourier expansion
\begin{equation}
f_r(\tau) \,=\, \sum_{D=1}^{\infty} b_r(D)\, q^D,
\end{equation}
and recall that
\begin{equation}
\eta(\tau)^r \,=\, q^{\frac{r}{24} }\sum_{n=0}^{\infty} a_r(n) \, q^n.
\end{equation}
\begin{proposition}
\label{ketten}Let $ 1 \leq r < 24$ be an odd integer. Let $n_0 \in \mathbb{N}$ be given, such that $D_0 := 24 n_0 + r$ satisfies
$p^2 \nmid D_0$ for all prime numbers $p \neq 2,3$. Let
\begin{equation}
\mathcal{N}_r(n_0):=
\left\{ n_0 \,l^{2}+r \left( l^{2}-1\right) /24 \,\,{\large|}\,\, l \in \mathbb{N}, \,\,(l,2 \cdot 3)=1      \right\}.
\end{equation}
Suppose that $a_r(n_0)=0$. Then $a_r(n) = 0$ for all $n \in \mathcal{N}_r(n_0)$. We call such numbers $n_0$ sources.
\end{proposition}
\begin{proof}

We can apply Lemma \ref{eigenforms}, since $\eta(24 \tau)^r$ is an Hecke eigenform \cite{Li57}.
Using the simple translation
$$ b_r(D) = a_r \left( \frac{D-r}{24}\right) \mbox{ and } a_r(n) = b_r \left( 24n + r \right)$$
leads to the result.
\end{proof}
If we impose additional conditions $3 \mid r$ and $27 \nmid D_0$ for the source $n_0$ then we get a stronger result, that $a_r(n)=0$ for all elements of
\begin{equation}
\left\{ n_0 \,l^{2}+r \left( l^{2}-1\right) /24 \,\,{\large|}\,\, l \in \mathbb{N}, \,\,(l,2)=1      \right\}.
\end{equation}
We refer to \cite{HNR17} for more details. For later use, note that for $r=15$ (since $27 \nmid D_0$), we obtain (substituting $2l+1$ for $l$ in the above)
$$\mathcal{N}_{15}(53)= \left\{ 53 + 429 \frac{ l (l+1)}{2} \, \vert \, l \in \mathbb{N}_0 \right\}.$$

\subsection{\label{extSerre}Serre's table extended}
Let $r $ be an odd integer. Serre (see introduction) recorded pairs $(r,n)$ for which $a_r(n)=0$, based on partly
(un)published work of Atkin, Newman, and Cohen.
In Serre's table, pairs for $r=5,7,15$ appear. We checked that for all pairs $(r,n_0)$ in this table, $n_0$ is a source.
Since in these cases the underlying eta-product is an eigenform, each of these sources $n_0$
implies the existence of infinitely many other $n\in \mathbb{N}$ such that $a_r(n)=0$.
For convenience, let $a_r(n)=0$ and $n$ a source, then we denote $(r,n)$ a source pair.

Motivated by \cite{HNR17}, where no further sources were found for $ 7 \leq r \leq 23$ for $n \leq 50 000$,
we conducted an intensive search for new sources. This was done in the $n$ aspect, especially for $r=7,9,11,13,15$ and in the $r$ aspect for
$n \leq 10^{8}$. No new source pair was found (except for $r=5$, where there are many). We believe  that this discovery is of general interest, shedding new light on Serre's  table and on other results presented in the literature on this topic.

\begin{theorem}
Let $r=7,9,11,13,15$.
Then there exist among all possible pairs $(r,n)$ with $a_r(n)=0$, for $n\leq 10^{10}$, exactly two source pairs $(7, 28017), (15, 53)$.
\end{theorem}
\begin{theorem}
Let $r$ be odd. Let $17 \leq r \leq 27$.
Then there exists no pair $(r,n)$ such that $a_r(n)=0$, for $n\leq 10^9$.
Let $29 \leq r \leq 550$. 
Then there exists no pair $(r,n)$ such that $a_r(n)=0$, for $n\leq 10^8$.
\end{theorem}
\ \\ \
\renewcommand{\arraystretch}{1.5}
{\bf Serre's table extended.}  \ \\ \\
\begin{tabular}{|c|c|c|c|}
  \hline
  $r$ & Sources $n_0$ & $\mathcal{N}_r(n_0)$ &  checked up to \\
  \hline \hline
  5 & $1560,1802,\ldots $ & $\{n_{0} l^{2}+5\cdot
                            \frac{l^{2}-1}{24}$, $(l,2\cdot3)=1$, $l\in \mathbb{N}\}$ & $10^{10}$ \\
  \hline
  7 & 28017 & $ \{ 28017 \, l^{2}+7
              \frac{l^{2}-1}{24}$,
              $(l,2\cdot3)=1$, $l\in \mathbb{N} \}$
                                             & $10^{10}$ \\ \hline
  9 & -- & $\emptyset$ & $10^{10}$ \\ \hline
  11 & -- &  $\emptyset$ & $10^{10}$ \\ \hline
  13 & -- & $\emptyset$ & $10^{10}$ \\ \hline
  15 & 53 & $ \{429\binom{l}{2}+53$,
            $l\in \mathbb{N}
 \}$
                                             & $10^{10}$ \\ \hline
  $17 \leq r \leq 27$ &  -- & $\emptyset$ & $10^{9}$ \\ \hline
  $29 \leq r \leq 549$ &  -- &$\emptyset$  & $10^{8}$ \\ \hline
\end{tabular}
\ \\ \\
For $r=5$ we have the following distribution of sources.

\[
\begin{array}{|c|c|c|c|c|c|c|c|c|}
\hline
X & 10^3 & 10^4 & 10^5 & 10^6 & 10^7 & 10^8 & 10^9 & 10^{10}\\ \hline
\text{Number of sources } n_0\leq X & 0 & 19 & 70 & 235 & 579 & 1402 & 3052 & 6352 \\ \hline
\end{array}
\]
 \ \\
The strategy for calculating the Fourier coefficients of
$\eta(\tau)^r$ is straightforward. For $r=1$ and $r=3$ we have the
exact formulas due to Euler and Jacobi,
\eqref{Euler} and \eqref{Jacobi}. Other powers $r$ can then be obtained by
exponentiating and multiplying those two series, taking into
account only terms that can contribute to a truncated series of given
order $n$. Such an algebra of truncated power series is implemented in
many computer algebra systems. We first used pari/gp for the
calculations, but then noticed that the flint
library~\cite{Hart2010,flint} is much faster. Moreover the flint
library has a built-in function for calculating the series expansion
of $\eta^r$, and recent development versions use multiple CPU cores for
polynomial multiplication. For example, on one of our workstations the
evaluation of $\eta^9$ up to order $n=10^8$ takes 20 min 12 s with
pari/gp and 1 min 56 s with flint (using only one core; with multiple
cores it takes less than 1 min). The limiting factor in the
calculations is random access memory.  For extremely high orders
$n=10^9 \ldots 10^{10}$ we used a high end server with a main memory
of 3~TB.

\subsection{Questions of Cohen and Str{\"o}mberg}

Cohen and Str{\"o}mberg (\cite{CS17}, Exercise 2.6) made the following
conjectures:
\begin{enumerate} \item The Fourier expansions of $\eta^5$ and of $\eta^{15}$ have
infinitely many zero coefficients, perhaps even more than $X^{\delta}$ up
to $X$ for some $\delta >0$ (perhaps any $\delta < 1/2$).
\item The Fourier
expansion of $\eta^7$ has infinitely many zero coefficients, perhaps of
order $\log(X)$ up to $X$. \end{enumerate} We also refer to Ono (\cite{On03}, Problem 3.51).

The $\eta^r$ for $r=5,7,15$ are Hecke eigenforms. Further, in all cases
sources exist. Hence there are infinitely many pairs $(r,n)$ such that $%
a_r(n)=0$.

We can answer both problems of Cohen and Str\"{o}mberg in the following way.
For a function $X\mapsto f\left( X\right) $
we use the Landau notation that it is $\Omega \left( g\left( X\right) \right) $
if $\limsup _{X\rightarrow \infty }\left| f\left( X\right) /g\left( X\right) \right| >0$.

\begin{proposition}
\label{cohenstromberg}The Fourier
expansions of $\eta ^{r}$ for $r=5,7,15$ have $\Omega \left(
X^{1/2}\right) $ coefficients that vanish.

More precisely, the following holds 
for arbitrary~$X$ with restrictions below.

\begin{enumerate}
\item  Let $r=5$. If $X\geq 7903\,77629
$ then
the number of indices $n\leq X$ for which the $n$th coefficient
$a_{r}\left( n\right) $ vanishes 
is larger
than $\frac{1}{119}\,X^{1/2}$.

\item  Let $r=7$. If $X\geq 10^{10}$ then
the number of indices $n\leq X$ for which the $n$th coefficient
$a_{r}\left( n\right) $ vanishes 
is larger
than $\frac{1}{505
}\,X^{1/2}$.

\item  Let $r=15$. If $X\geq 96183
$ then
the number of indices $n\leq X$ for which the $n$th coefficient
$a_{r}\left( n\right) $ vanishes 
is larger
than $\frac{1}{15}\,X^{1/2}$.

\end{enumerate}
\end{proposition}

\begin{proof}
This follows readily from Proposition~\ref{ketten} and the extended table of Serre.

For each case $r=5,7,15$ we get at least one source pair.
From the proposition we then get a quadratically growing sequence of
indices where coefficients vanish. This implies that $\Omega
\left( X^{1/2}\right) $ of the coefficients up to a number~$X$ vanish.

\begin{enumerate}
\item  For $r=5$ we know that there is a source at $n=1560$ and that the
coefficients vanish 
at the indices $n=1560\left( 6l\pm 1\right)
^{2}+5\left( \left( 6l\pm 1\right) ^{2}-1\right) /24$. This is less than $X$
if $l\leq \left( \sqrt{\left( 24X+5\right) /37445}\mp 1\right) /6$. In total
we obtain at least $\sqrt{\left( 24X+5\right) /37445}/3-1
$ coefficients that
vanish. If we divide by $\sqrt{X}$ we obtain $\sqrt{\left( 24+5/X\right)
/37445}/3-1
/\sqrt{X}\geq \sqrt{24
/37445}/3-1/\sqrt{X}$ 
for all $X>0$ 
and for
$X\geq Y=7903\,77629
$ we obtain
$\sqrt{
24
/37445}/3-1
/\sqrt{X}\geq
\sqrt{
24
/37445}/3-1
/\sqrt{Y}>
\frac{1}{119}$.

\item  For $r=7$ we know that the coefficients at indices $n=28017\left(
6l\pm 1\right) ^{2}+7\left( \left( 6l\pm 1\right) ^{2}-1\right) /24$ vanish. Then $n\leq X$ if
\[
l\leq \frac{1}{6}\left( \sqrt{\frac{24X+7}{672415}}\mp 1\right) .
\]
Hence in total there are at least $\sqrt{\left( 24X+7\right) /672415}/3-1$
many such $l$s and if we divide this by $\sqrt{X}$ we obtain $\sqrt{\left(
24+7X^{-1}\right) /672415}/3-X^{-1/2}\geq \sqrt{
24/672415}/3-X^{-1/2}
>1/505$ for $X\geq 10^{10}$.

\item  Our computations up to $10^{10}$ show that $a_{15}\left( n\right) =0$ if and only
if
\begin{equation}
n=53+\frac{429}{2}l\left( l-1\right) =\frac{429}{2}\left( l-\frac{1}{2}%
\right) ^{2}-\frac{5}{8}  \label{eq:quadrat}
\end{equation}
for some $l\in \mathbb{N}
$. 
Hence there are more
than $L
=\sqrt{2\left(
X+5/8\right) /429}-
1/2$ coefficients zero. Hence $L
/\sqrt{X}\geq \sqrt{%
\left( 2+5X^{-1}/4\right) /429}-X^{-1/2}/2\geq \sqrt{%
2/429}-X^{-1/2}/2$ 
and $\sqrt{
2
/429}-X^{-1/2}/2>1/15$ for $X\geq 96183
$.
\end{enumerate}
\end{proof}

\begin{remark}
The numerical data up to $n=10^{10}$ for $r=5$ seems to suggest that there
is even a~$\delta $ larger than $1/2$ such that $\Omega \left( X^{\delta
}\right) $ coefficients are zero.
\end{remark}

\subsection{A question of Ono}
Ono (\cite{On03}, Problem 3.51) asked a question complementary to that of Cohen and Str\"{o}mberg, on {\em non}-vanishing coefficients for odd $r \geq 5$.
In contrast to our results towards the Cohen and Str\"{o}mberg conjectures, we can give only partial answers to Ono's questions.

We only have to give estimates for $r=5,7$ and $r=15$.

\begin{remark}
For $r=5$, numerical computations show that for $X\leq 10^{10}$,
\[
\left| \left\{ n\leq X:a_{
5}\left( n\right) \neq 0\right\} \right| \geq
\frac{2598}{2605}X>0.9973X.
\]
\end{remark}

\begin{proposition}
\label{ono15}For $X\leq 10^{10}$ we have the following.

\begin{enumerate}
\item  For $r=7$ holds
\[
\left| \left\{ n\leq X:a_{7}\left( n\right)
\neq 0\right\} \right| \geq
\frac{84047}{84051}X>0.99995X
\]
and 
\[
\left| \left\{ n\leq X:a_{7}\left( n\right) \neq 0\right\} \right| >\left(
1-\left( 125X^{1/2}\right) ^{-1}\right) X.
\]

\item  For $r=15$ holds
\[
\left| \left\{ n\leq X:a_{15}\left( n\right) \neq 0\right\} \right| \geq
\frac{52}{53}X>0.98113X
\]
and if in addition 
$X\geq 25214
$:
\[
\left| \left\{ n\leq X:a_{15}\left( n\right) \neq 0\right\} \right| >\left(
1-\left( 14X^{1/2}\right) ^{-1}\right) X.
\]
\end{enumerate}
\end{proposition}

\begin{proof}
\begin{enumerate}
\item  Our computations up to $10^{10}$ show that $a_{7}\left( n\right) =0$ only
if
\begin{equation}
n=28017\left( 6l\pm 1\right) ^{2}+\frac{7}{24}\left( \left( 6l\pm 1\right) ^{2}-1\right)   \label{eq:quadrat7}
\end{equation}
for some $l\in \mathbb{N}$ and for
`$+$' also $l=0$ is allowed.
This is less than $X$ if and only if
$0\leq l\leq \frac{1}{6}\left( \left( \frac{24X+7}{672415} \right) ^{1/2}\mp %
1\right) $
where $l=0$ only in case `$-$' is
allowed. In total we obtain at most
\begin{equation}
\frac{1}{3}\sqrt{ \frac{24X+7}{672415}}+%
1  \label{eq:krange7}
\end{equation}
such $l$ for $n\leq  X$.
For $X\leq 28016$ there are no such coefficients so any
positive upper bound will do.
So we assume $X\geq 28017$ and
by (\ref{eq:krange7}) 
an upper bound on the portion of vanishing coefficients
up to~$X$ 
is
\[
\frac{1}{3}\sqrt{ \frac%
{ 24X^{-1}+7X^{-2} }{%
672415}}+\frac{1}{X}\leq \frac{4}{84051}.
\]
This implies the first claim since 
then 
\begin{eqnarray*}
&&\left| \left\{ n\leq X:a_{7}\left( n\right)
\neq 0\right\} \right|  \\
&=&X-\left| \left\{ n\leq X:a_{7}\left( n\right)
=0\right\} \right| =\left( 1-\left| \left\{ n\leq X:a_{7}\left( n\right)
=
0\right\} \right| /X\right) X\\
&\geq &
\left( 1-\left( \frac{1}{3}\sqrt{ \frac{24X+7}{672415}}+%
1\right) /X\right) X\geq \left( 1-\frac{4}{84051}\right) X
.
\end{eqnarray*}

The second claim follows from the quotient of the upper
bound (\ref{eq:krange7}) of the number of vanishing
coefficients up to~$X$ by $\sqrt{X}$ which is 
\[
\frac{1}{3} \sqrt{\frac{24+7X^{-1}}{672415}}+ \frac{1}{\sqrt{X}}
\leq \frac{1}{3}\sqrt{\frac{24+7/27699}{672415}}+\frac{1}{\sqrt{27699}}<\frac{1}{125}
\]
for $X\geq 27699$ and for $X\leq 27698<28017$ there are no
vanishing coefficients. 
Then 
\begin{eqnarray*}
&&\left| \left\{ n\leq X:a_{7}\left( n\right)
\neq 0\right\} \right| =X-\left| \left\{ n\leq X:a_{7}\left( n\right)
=0\right\} \right| \\
&=&\left( 1-\left( 1/\sqrt{X}\right)
\left( \left| \left\{ n\leq X:a_{7}\left( n\right)
=
0\right\} \right| /\sqrt{X}\right) \right) X\\
&\geq &
\left( 1-\left( 1/\sqrt{X}\right) \left( \frac{1}{3}\sqrt{ \frac{24X+7}{672415}}+%
1\right) /\sqrt{X}\right) X>
\left( 1-\frac{1}{125\sqrt{X}}\right) X
.
\end{eqnarray*}

\item 
For $r=15$ we observe 
in the proof of
Proposition~\ref{cohenstromberg} that (\ref{eq:quadrat}) 
is less than $X$ if and only if
\begin{equation}
1\leq l\leq \left( \frac{2}{429}\left( X+\frac{5}{8}\right) \right) ^{1/2}+%
\frac{1}{2}.  \label{eq:krange}
\end{equation}
So there are less 
than $L
=\sqrt{2\left(
X+5/8\right) /429}+
1/2$ coefficients zero.
There is no such coefficient 
for $X\leq 52$. 
Hence we can assume $X\geq 53$. By
(\ref{eq:krange}) an upper bound for the portion of vanishing
coefficients 
is given by $L
/X\leq \left( \frac{2}{%
429}\left( X^{-1}+5X^{-2}/8\right) \right) ^{1/2}+X^{-1}/2\leq \frac{1}{53}$%
. This implies the first claim 
since then 
\begin{eqnarray*}
&&\left| \left\{ n\leq X:a_{15}\left( n\right)
\neq 0\right\} \right| =X
-\left| \left\{ n\leq X:a_{15}\left( n\right)
=0\right\} \right| \\
&=&\left( 1-\left| \left\{ n\leq X:a_{15}\left( n\right)
=
0\right\} \right| /X\right) X\\
&\geq &
\left( 1-\left( \left( \frac{2}{429}\left( X+%
\frac{5}{8}\right) \right) ^{1/2}+%
\frac{1}{2}\right) /X\right) X
\geq 
\left( 1-\frac{1}{53}\right) X
.
\end{eqnarray*}

The second claim follows from the quotient of the upper
bound (\ref{eq:krange}) of the number of vanishing
coefficients up to~$X$ by $\sqrt{X}$ which is 
\begin{eqnarray*}
\frac{L
}{\sqrt{X}}&\leq &\sqrt{\frac{2}{429}\left( 1+\frac{5}{8X}\right) }+\frac{1}{2\sqrt{X}}\\
&\leq &
\sqrt{\frac{2}{429}\left( 1+\frac{5}{8\cdot 25214}\right) }+\frac{1}{2\sqrt{25214}}<\frac{1}{14}
\end{eqnarray*}
for $X\geq 25214$ since then 
\begin{eqnarray*}
&&\left| \left\{ n\leq X:a_{15}\left( n\right)
\neq 0\right\} \right| %
=X-\left| \left\{ n\leq X:a_{15}\left( n\right)
=0\right\} \right| \\
&=&\left( 1-\left( 1/\sqrt{X}\right)
\left( \left| \left\{ n\leq X:a_{15}\left( n\right)
=
0\right\} \right| /\sqrt{X}\right) \right) X\\
&\geq &
\left( 1-\left( 1/\sqrt{X}\right) \left( \left( \frac{2}{429}\left( X+\frac{5}{8}\right) \right) ^{1/2}+%
\frac{1}{2}\right) /\sqrt{X}\right) X\\
& > & \left( 1-\frac{1}{14\sqrt{X}}\right) X\,.
\end{eqnarray*}
\end{enumerate}
\end{proof}


\section{Even powers of the Dedekind eta function}
In this section we consider the vanishing
properties of Fourier coefficients of
even powers $r$ of the Dedekind eta function $\eta$.

We cited Serre's \cite{Se85} famous result on the characterization of the lacunary $\eta^r$.
Hence let $r \in \mathbb{N}$ be even and let $r \not\in S_{\text{even}}=  \{ 2,4,6,8,10,14,26\}$ throughout this section.
The following result is quite surprising.

\begin{theorem}
\label{general}
Let $r$ an even positive integer. If $r \not\in S_{\text{even}}$, with
$ 12 \leq r \leq 132$, then $a_r(n) \neq 0$ for $n \leq 10^8$.
If $ 124 \leq r \leq 550$, then $a_r(n) \neq 0$ for $n \leq 10^7$.
\end{theorem}
The result is obtained by numerical computations. See also the previous section. \\
It might lead one to speculate that there exists an $n \in \mathbb{N}$ such that $a_r(n)=0$ iff $r \in S_{\text{even}}=  \{ 2,4,6,8,10,14,26\}$.
This would include Lehmer's conjecture on the discriminant function \cite{Le47} as the very special case $r=24$.

We also show that the case $r=48$, on the square of the discriminant function $\Delta$,
is closely connected to a conjecture by Maeda, although in this case we are not dealing with an Hecke eigenform.

Let us first recall the link to modular forms of integral weight $k=r/2$ with Nebentypus, and fix some notation.




\subsection{Modular forms of integral weight and Hecke theory}

Let $N$ be a positive integer, and $\chi$ be a Dirichlet character modulo $N$.
Let $k$ be an integer. Then we denote by $S_k(\Gamma_0(N), \chi)$ the $\mathbb{C}$-vector space of
modular forms with respect to $\Gamma_0(N)$ of weight $k$ and Nebentypus $\chi$. If $\chi$ is trivial we write $S_k(\Gamma_0(N))$ for short. Further we put $S_k(\Gamma)=S_k(\Gamma_0(1))$. Here $\Gamma := SL_2(\mathbb{Z})$.
Cusp forms $f$ have Fourier expansions
\begin{equation}
f(\tau) = \sum_{n=1}^{\infty} b(n) \, q^n.
\end{equation}

We are mainly interested  when the special eta-products $\eta^r(\delta \, \tau)$ are cusp forms.
Therefore we specialize Theorem 1.64 \cite{On03} due to Gordon, Hughes and Newman to our situation.
\begin{corollary}
Let $f(\tau):= \eta^r(\delta \, \tau)$ with $k:= \frac{r}{2}$
a positive integer and $N, \, \delta \in \mathbb{N}$.
Suppose that
\begin{eqnarray}
\delta \, r  & \equiv & 0 \pmod{24}, \\
\frac{N}{\delta} \, r & \equiv & 0 \pmod{24}.
\end{eqnarray}
Then $f \in S_k(\Gamma_0(N), \chi)$. Here $\chi(d):= \left( \frac{(-1)^k s}{d} \right)$ and $s:= \delta^r$.
\end{corollary}
We deduce that
$$ \eta^{12}(2 \tau) \in S_6(\Gamma_0(4)) \mbox{ and } \Delta(\tau):= \eta^{24}(\tau) \in S_{12}(\Gamma).$$
\begin{definition}
Let $f \in S_k(\Gamma_0(N), \chi)$. Then we define Hecke operators $T_{m,k,\chi}$ by
\begin{equation}
T(f):= \sum_{n=1}^{\infty} \left( \sum_{d \, (m,n)} \chi(d) \, d^{k-1} b(mn/d^2) \right) q^n.
\end{equation}
\end{definition}
We say that $f$ a normalized Hecke eigenform if $f$ is an eigenform for all Hecke operators $T_{p,k,\chi}$, where $p$ is a prime and $p \nmid N$,
and $b(1)=1$.


\subsection{The discriminant function $\Delta$ and Lehmer's conjecture}
The smallest positive integer $r$ with the property that $\eta^r$
is a modular cusp form for the full modular group $\Gamma$ is $r=24$. It gives the discriminant function

\begin{equation}
\Delta(\tau) := q \prod_{n=1}^{\infty} \left( 1 - q^n \right)^{24} = \sum_{n=1}^{\infty} \tau(n) \,\, q^n.
\end{equation}
Note that in our notation $\tau(n) := a_{24}(n-1)$ is called the Ramanujan function.
Since $\dim S_{12}(\Gamma) =1$, $\Delta$ is also a Hecke eigenform. This implies that the Fourier coefficients $\tau(n)$ are also the
eigenvalues of the $n$-th Hecke operator. This implies for example that $\tau(n \, m ) = \tau(n) \, \tau(m)$ for $n$ and $m$ coprime.
Triggered by numerical evidence, Lehmer \cite{Le47} conjectured in 1947, that $\tau(n) \neq 0$ for all $n \in \mathbb{N}$.
He observed also that the smallest possible $n$ with $\tau(n)=0$ has to be a prime number. The conjecture is still open.
Serre \cite{Se81} has proven that the set of primes for which $\tau(p)$ vanishes has density $0$ (in the set of all primes).

It is not clear or obvious what a possible generalization of Lehmer's conjecture could be, in which setting is should belong.
For example, Ono \cite{On95}  studied  $\eta^{12}(2 \tau) \in S_6(\Gamma_0(4))$, which is also a Hecke eigenform since the underlying vector space has dimension one.
There is no natural number known such that $a_{12}(n)=0$. Also, the smallest $n$ such that $a_{12}(n-1)=0$ would be a prime number.

One could also look at all spaces $S_k(\Gamma)$ with dimension one, i.e.  $S_k(\Gamma) = \mathbb{C} \, g_k$. Here $g_k$ is a normalized cusp form, which is also a Hecke eigenform. There are only finitely many such weights $k$, i.e. $12,16,18,20,22$ and $26$.
Let $a_k(n)$ be the Fourier coefficients of $g_k$. Then conjecturally $a_k(n) \neq 0$ for all $n\in \mathbb{N}$. Currently this generalized conjecture is still open (see also Bruinier (\cite{Br02}, 3.4.1) for an reinterpretation by Freitag). For these eigenforms, the non-vanishing of the $n$-th eigenvalue was studied by \cite{DHZ14} involving Galois representations
and sophisticated calculations. This led them to the outstanding record that Lehmer's conjecture is true for $n \leq N \approx 8 \, \cdot \, 10^{23}$. Their method does not apply for powers of the Dedekind eta function in general; they have to be Hecke eigenforms.
Thus it does not lead us to expect the observations in Theorem \ref{general}, since
for example $\Delta^d = \eta^{24 d}$ are not Hecke eigenforms for $ d >1$. But, striking as those observations might be, perhaps it would be too bold of us to generalize Lehmer's conjecture to all even powers of the Dedekind eta function that are not CM-forms.


\subsection{$\Delta^2$ and Maeda's conjecture}
We have already observed that all coefficients $a_{48}(n)$ for $n \leq 10^{8}$ are non-vanishing.
The cusp form $\eta^{48} = \Delta^2 \in S_{24}(\Gamma)$ is not a Hecke eigenform.
Note that $\mbox{dim}S_{24}(\Gamma)=2$. In the following we show that Maeda's conjecture \cite{HM97} supports our results by implying
that for all $n \in\mathbb{N}$ the coefficients $a_{48}(n)\neq 0$.
Let us first give us some extension of our previous calculations in the case $r=48$.
\begin{theorem}
Let $a_{48}(n)$ be the Fourier coefficients of $\Delta^2$
\begin{equation}
\Delta^2(\tau) = \eta^{48}(\tau) = q^2 \sum_{n=0}^{\infty} a_{48}(n) \, q^n.
\end{equation}
Then for $n \leq 5\cdot 10^{9}$ all coefficients are different from zero.
\end{theorem}

{\bf Maeda's conjecture \cite{HM97}, \cite{GM12}.} Let $n>1$ be a positive integer. We consider the characteristic polynomial of the
action of the Hecke operator $T_n$ on $S_k(\Gamma)$. Maeda conjectured that this characteristic polynomial is irreducible over $\Q$, and
that the Galois group of the splitting field is the full symmetric group of size $d!$, where $d = \dim S_k(\Gamma)$.
In particular, all the eigenvalues of $T_n$ on $S_k(\Gamma)$ should be distinct.

Let $f$ and $g$ be the two normalized Hecke eigenforms of $S_{24}(\Gamma)$. Then there exists a constant $\kappa \neq 0$ such that
$$ \Delta^2 = \frac{f -g}{\kappa}.$$
The Hecke field generated by the Fourier coefficients of $f$ and $g$ is given by $$\mathbb{Q}(\sqrt{144169}).$$
Let $f(\tau) = \sum_{n=1}^{\infty} a_f(n) \,\, q^n$ and $g(\tau) = \sum_{n=1}^{\infty} a_g(n) \,\, q^n$.
Then $a_f(n) = A(n) + B(n)\, \,  \sqrt{144169}$ and $a_g(n) = A(n) - B(n)\, \,  \sqrt{144169}$
with $A(n),B(n) \in \mathbb{Z}$. Note, that $a_f(n)$ and $a_g(n)$ are also the Hecke eigenvalues.
See also \cite{DG96},\cite{KK07}, where this example had been worked out.

It follows that the $n^{\mathrm{th}}$ Fourier coefficient of $\Delta^2$ is non-zero if and only if
the eigenvalues of $T_n$ on $S_{24}(\Gamma)$ are distinct. Hence it is expected from Maeda's point of view that all Fourier coefficients of
$\Delta^2$ are different from zero. Finally we note that the previous record for checking Maeda's conjecture in this case is $n \leq 10^5$ (\cite{GM12}). Our result extends this to
$n \leq 5 \,\cdot \, 10^9$.


 \end{document}